\documentclass[a4,12pt]{amsart}

\usepackage{amssymb}
\usepackage{amstext}
\usepackage{amsmath}
\usepackage{amscd}
\usepackage{amsthm}
\usepackage{amsfonts}
\usepackage{enumerate}
\usepackage{graphicx}
\usepackage{latexsym}
\usepackage[all]{xy}
\usepackage[usenames]{color}

\newtheorem*{corollary*}{Corollary}
\newtheorem{theorem}{Theorem}[section]

\newtheorem{corollary}[theorem]{Corollary}
\newtheorem{lemma}[theorem]{Lemma}
\newtheorem{proposition}[theorem]{Proposition}

\newtheorem*{claim*}{Claim}

\theoremstyle{definition}
\newtheorem{definition}[theorem]{Definition}
\newtheorem{remark}[theorem]{Remark}
\newtheorem{example}[theorem]{Example}
\newtheorem{notation}[theorem]{Notation}

\numberwithin{equation}{theorem}

\def\mod{\operatorname{mod}}
\newcommand{\proj}{\operatorname{proj}} 
\newcommand{\CM}{\operatorname{CM}}
\def\sCM{\operatorname{\underline{CM}}}

\def\id{\operatorname{inj.dim}}

\def\gldim{\operatorname{gl.dim}}
\newcommand{\tridim}{\operatorname{tri.dim}}
\newcommand{\Xtridim}[1]{#1\mbox{-}\tridim}

\def\Db{\mathsf{D}^{\mathrm{b}}}

\newcommand{\C}{\mathsf{C}}

\newcommand{\add}{\operatorname{add}}

\newcommand{\res}{\operatorname{res}}

\def\Hom{\operatorname{Hom}}
\newcommand{\End}{\operatorname{End}}
\def\Ext{\operatorname{Ext}}

\newcommand{\Ker}{\operatorname{Ker}}

\newcommand{\gen}[3]{\langle #1{\rangle}_{#2}^{#3}}
\renewcommand{\AA}{{\mathcal A}}

\newcommand{\CC}{{\mathcal C}}

\newcommand{\TT}{{\mathcal T}}

\newcommand{\XX}{{\mathcal X}}
\newcommand{\YY}{{\mathcal Y}}
\newcommand{\FF}{{\mathcal F}}
\def\ZZ{\mathcal{Z}}

\newcommand{\Z}{\mathbb{Z}}

\def\Spec{\operatorname{Spec}}
\def\p{\mathfrak{p}}

\def\m{\mathfrak{m}}

\def\LoL{\operatorname{\ell\ell}}

\begin{document}
\setlength{\baselineskip}{15pt}
\title[Dimensions with respect to subcategories]{Dimensions of triangulated categories\\
with respect to subcategories}
\author[T. Aihara]{Takuma Aihara}
\address{(T. Aihara) Division of Mathematical Science and Physics, Graduate School of Science and Technology, 
Chiba University, Yayoi-cho, Chiba 263-8522, Japan}
\curraddr{Graduate School of Mathematics, Nagoya University, Furocho, Chikusaku, Nagoya 464-8602, Japan}
\email{aihara.takuma@math.nagoya-u.ac.jp}
\author[T. Araya]{Tokuji Araya}
\address{(T. Araya) Liberal Arts Division, Tokuyama College of Technology, Gakuendai, Shunan, Yamaguchi, 745-8585, Japan}
\curraddr{Department of Applied Science, Faculty of Science, Okayama University of Science, Ridaicho, Kitaku, Okayama 700-0005, Japan}
\email{araya@das.ous.ac.jp}
\author[O. Iyama]{Osamu Iyama}
\address{(O. Iyama) Graduate School of Mathematics, Nagoya University, Furocho, Chikusaku, Nagoya 464-8602, Japan}
\email{iyama@math.nagoya-u.ac.jp}
\urladdr{http://www.math.nagoya-u.ac.jp/~iyama/}
\author[R. Takahashi]{Ryo Takahashi}
\address{(R. Takahashi) Graduate School of Mathematics, Nagoya University, Furocho, Chikusaku, Nagoya 464-8602, Japan/Department of Mathematics, University of Nebraska, Lincoln, NE 68588-0130, USA}
\curraddr{Graduate School of Mathematics, Nagoya University, Furocho, Chikusaku, Nagoya 464-8602, Japan}
\email{takahashi@math.nagoya-u.ac.jp}
\urladdr{http://www.math.nagoya-u.ac.jp/~takahashi/}
\author[M. Yoshiwaki]{Michio Yoshiwaki}
\address{(M. Yoshiwaki) Osaka City University Advanced Mathematical Institute, 3-3-138 Sugimoto, Sumiyoshi-ku, Osaka 558-8585, Japan}
\email{yosiwaki@sci.osaka-cu.ac.jp}
\thanks{The second author was partially supported by JSPS Research Activity Start-up 23840043. The third author was partially supported by JSPS Grant-in-Aid for Scientific Research 21740010, 21340003, 20244001 and 22224001. The fourth author was partially supported by JSPS Grant-in-Aid for Young Scientists (B) 22740008 and by JSPS Postdoctoral Fellowships for Research Abroad}
\begin{abstract}
This paper introduces a concept of dimension of a triangulated category with respect to a fixed full subcategory.
For the bounded derived category of an abelian category, upper bounds of the dimension with respect to a contravariantly finite subcategory and a resolving subcategory are given.
Our methods not only recover some known results on the dimensions of derived categories in the sense of Rouquier, but also apply to various commutative and non-commutative noetherian rings.
\end{abstract}
\maketitle
\section{Introduction}

The notion of dimension of a triangulated category was introduced by Rouquier \cite{R2} based on work of Bondal and Van den Bergh \cite{BV} on Brown representability.
It measures how many extensions are needed to build the triangulated category out of a single object, up to finite direct sum, direct summand and shift.
For results on the dimensions of triangulated categories, we refer to \cite{AT,BIKO,CYZ,KK,Yoshiw} for instance.

It is still a hard problem in general to give a precise value of the dimension of a given triangulated category.
The aim of this paper is to provide new information on this problem.
We give upper bounds of the dimensions of derived categories in terms of global dimensions.
A prototype of our approach is given by the inequality 
\begin{equation}\label{prototype}
\tridim\Db(\mod\Lambda)\leq\gldim\End_\Lambda(M)
\end{equation}
for a noetherian algebra $\Lambda$ and a generator $M$ of $\Lambda$ \cite[(3.4)]{KK}, 
where $\tridim\TT$ denotes the dimension of a triangulated category $\TT$.
This observation was applied to study representation dimension \cite{HIO,KK,O3,R1}.

Let $\TT$ be a triangulated category and $\XX$ a full subcategory.
This paper introduces and studies {\em the dimension
$$
\Xtridim{\XX}\TT
$$
of $\TT$ with respect to $\XX$}, which measures how many extensions are needed to build $\TT$ out of $\XX$, up to finite direct sum, direct summand and shift.
A similar notion called {\em level} was studied by Avramov, Buchweitz, Iyengar and Miller \cite{ABIM}; it is defined for each object of $\TT$.

In this paper, first we generalize the inequality \eqref{prototype} by replacing the right hand side with the global dimension of a functor category.
In representation theory, to study a category $\XX$ of representations of $\Lambda$ such as $\mod\Lambda$, $\CM(\Lambda)$ and $\Db(\mod\Lambda)$, the functor category $\mod\XX$ plays a crucial role \cite{AR}.
For example, it follows from a basic observation in Auslander-Reiten theory that projective resolutions of simple objects in $\mod\XX$ correspond to almost split sequences in $\XX$ \cite{ASS,Y}.
Our first main result is the following theorem.

\begin{theorem}\label{main result}
Let $\AA$ be an abelian category, and $\XX$ a contravariantly finite subcategory that generates $\AA$. 
Then
\[\Xtridim{\XX}\Db(\AA)\leq\gldim(\mod\XX).\]
\end{theorem}

\noindent
One can recover \eqref{prototype} by letting $\XX=\add M$.
For a cotilting module $T$ we apply this result to the full subcategory $\XX_T$ of $\mod\Lambda$ consisting of modules $X$ with $\Ext_\Lambda^{>0}(X,T)=0$ to get:

\begin{corollary}\label{main corollary}
Let $\Lambda$ be a noetherian ring and $T$ a cotilting module.
Then it holds that $\Xtridim{\XX_T}\Db(\mod\Lambda)\le\max\{1,\id T\}$.
\end{corollary}

\noindent
For example, when $R$ is a Cohen-Macaulay local ring with a canonical module, the above result induces an inequality
\begin{equation}\label{intro:CM}
\Xtridim{\CM(R)}\Db(\mod R)\leq\max\{1,\dim R\}.
\end{equation}
In particular, if $R$ has finite Cohen-Macaulay representation type, then the inequality $\tridim\Db(\mod R)\leq\max\{1,\dim R\}$ holds.

Next, we give another approach based on Cartan-Eilenberg resolutions in derived categories.
We prove the following theorem as our second main result.

\begin{theorem}\label{main2}
Let $\AA$ be an abelian category with enough projective objects.
Let $d\ge0$ be an integer, and $\XX$ a resolving subcategory containing the $d$-th syzygies of objects in $\AA$. 
Then one has 
$$
\Xtridim{\XX}\Db(\AA)\leq\max\{1,d\}.
$$
\end{theorem}

\noindent
This result also recovers the inequality \eqref{intro:CM}, removing the assumption that $R$ possesses a canonical module.
Furthermore, for a reduced affine algebra it enables us to reproduce the Krull dimension from the dimension of a derived category with respect to a subcategory:

\begin{corollary}
Let $R$ be a finitely generated reduced commutative algebra over a field. 
Then there is an equality
\[\dim R=\Xtridim{\operatorname{FR}(R)}\Db(\mod R),\]
where $\operatorname{FR}(R)$ denotes the full subcategory of $\mod R$ consisting of modules that are locally free on the regular locus of $R$.
\end{corollary}

\noindent
This corollary recovers a result of Rouquier \cite{R2}: if $R$ is a polynomial ring over a field, then $\dim R=\tridim\Db(\mod R)$ holds.
The notion of dimension of a triangulated category has been introduced as an attempt to formulate the Krull dimensions of rings and schemes in terms of the triangulated structure of their derived categories.
However, it turned out that the dimension of a derived category is often different from the Krull dimension.
We hope that the notion of dimension with respect to a subcategory which we introduce in this paper can achieve that goal. 
However the above result is itself still insufficient since the definition of $\operatorname{FR}(R)$ requires information on localizations of $R$-modules.

This paper is organized as follows.
In Section \ref{section:preliminary}, we give preliminary results.
In Section \ref{section:functor category}, we recall basic properties of functor categories and give our first main Theorem \ref{main result}.
In Section \ref{section:resolving subcategory}, we give an approach based on Cartan-Eilenberg resolutions in derived categories and prove our second main Theorem \ref{main2}.
In Section \ref{section:application}, we give several applications for Iwanaga-Gorenstein rings, commutative Cohen-Macaulay rings and finitely generated commutative algebras over fields.

\section{Preliminaries}\label{section:preliminary}

This section is devoted to stating some basic definitions and preliminary results. 
Throughout this paper, we assume that all rings are associative, noetherian and with identity, and that all subcategories are full and closed under isomorphisms.
A triangle in a triangulated category always means an exact triangle.
The following notation will be used.

\begin{notation}
Let $\AA$ be an additive category.
We denote by $\proj\AA$ the subcategory of $\AA$ consisting of projective objects of $\AA$.
For a subcategory $\XX$ of $\AA$, we denote by $\add\XX$ the smallest subcategory of $\AA$ containing $\XX$ that is closed under finite direct sums and direct summands.
We denote by $\C^{\rm b}(\AA)$ the category of bounded complexes of objects of $\AA$.
When $\AA$ is abelian, the bounded derived category of $\AA$ is denoted by $\Db(\AA)$. 
For a ring $\Lambda$, let $\mod\Lambda$ denote the category of finitely generated right $\Lambda$-modules.
We simply write $\proj\Lambda$ instead of $\proj(\mod\Lambda)$.
For a commutative Cohen-Macaulay local ring $R$, we denote the category of (maximal) Cohen-Macaulay $R$-modules by $\CM(R)$.
\end{notation}

Let $\TT$ be a triangulated category and $\XX,\YY$ be subcategoies of $\TT$.
We denote by $\XX*\YY$ the subcategory of $\TT$ consisting of objects $M$ that admit triangles
$
X\to M\to Y\to X[1]
$
with $X\in \XX$ and $Y\in \YY$.
Then $(\XX*\YY)*\ZZ=\XX*(\YY*\ZZ)$ holds by octahedral axiom. 
Set $\gen{\XX}{}{}:=\add\{X[i]\ |\ X\in\XX, i\in\Z \}$. 
For a positive integer $n$, let 
\[\gen{\XX}{n}{\TT}=\gen{\XX}{n}{}:=\add(\underbrace{\gen{\XX}{}{}*\gen{\XX}{}{}*\cdots*\gen{\XX}{}{}}_{n}).\]
Clearly $\gen{\XX}{n}{}$ is closed under shifts.

The concept of dimension of a triangulated category was introduced by Rouquier \cite{R2}.
Now we define a version of Rouquier's dimension relative to a fixed full subcategory.

\begin{definition}\label{dtcat}
For a subcategory $\XX$ of $\TT$, 
we define 
\[\Xtridim{\XX}\TT:=\inf\{n\geq0\ |\ \TT=\gen{\XX}{n+1}{}\}.\]
\end{definition}

\begin{remark}
\begin{enumerate}
\item Our relative dimension has the following relationship with the notion of {\em level}, defined in \cite{ABIM}:
$$
\Xtridim{\XX}\TT=\sup_{M\in\TT}\left\{\mathsf{level}_\TT^\XX(M)\right\}-1.
$$

\item The {\em (triangle) dimension} $\tridim \TT$ defined by Rouquier is
described in terms of our relative dimension, as follows:
$$
\tridim \TT = \inf_{M\in\TT}\{\Xtridim{(\add M)}\TT\}.
$$
\end{enumerate}
\end{remark}


It is often that $\tridim\Db(\AA)$ is called the {\em derived dimension} of $\AA$.
For a ring $\Lambda$, we call the dimension $\tridim\Db(\mod\Lambda)$ the {\em derived dimension} of $\Lambda$.

The following easy observation will be used later.

\begin{lemma}\label{dense}
Let $F:\TT\to\TT'$ be a dense triangle functor of triangulated categories. 
Let $\XX$ be a subcategory of $\TT$ and $n\geq0$. 
If $\TT=\gen{\XX}{n}{}$, then $\TT'=\gen{F\XX}{n}{}$. 
\end{lemma}

Let $\AA$ be an abelian category with enough projective objects. 
As usual, the projective dimension of an object $M\in\AA$ is the infimum of the integers $d$ 
such that there exists an exact sequence $0\to P_d\to P_{d-1}\to\cdots\to P_0\to M\to0$
with $P_i\in\proj\AA$ for all $i$. 
We denote by $\gldim\AA$ the supremum of the projective dimensions of objects of $\AA$ and call it the {\em global dimension} of $\AA$.

The result below gives a relationship between global dimension and derived dimension, which follows from \cite[(2.4) and (2.5)]{KK}.

\begin{proposition}\label{KK}
Let $\AA$ be an abelian category with enough projective objects. 
Then one has $\Xtridim{(\proj\AA)}\Db(\AA)=\gldim\AA$.
\end{proposition}

We end this section with a useful information that sharpens a result in \cite{Han}.
Let $\AA$ be an abelian category with enough projective objects and $\CC$ a subcategory of $\AA$.
For a positive integer $n$, we denote by $\Omega^n\CC$ the subcategory of $\AA$ which consists of objects $X$ admitting an exact sequence $0\to X\to P_{n-1}\to\cdots\to P_0\to C\to0$ with $C\in\CC$ and $P_i\in\proj\AA$ for all $i$.
Note that $\Omega^n\CC$ contains $\proj\AA$.

\begin{proposition}\label{sub proj}
Let $\AA$ be an abelian category with enough projective objects. 
For every $M\in\Db(\AA)$, 
there exists a triangle $Z\to M\to B\to K[1]$ 
with $Z\in\gen{\Omega^2\AA}{}{}$ and $B\in\gen{\Omega\AA}{}{}$. 
In particular, one has $\Db(\AA)=\gen{\Omega\AA}{2}{}$. 
\end{proposition}

\begin{proof}
Let $M$ be an object of $\Db(\AA)$. 
We can regard $M$ as a right bounded complex $(\cdots\to P^{i}\xrightarrow{d^i}\cdots\xrightarrow{d^{b-1}}P^b\xrightarrow{d^b}0)$ of objects in $\proj\AA$. 
There exists an integer $a$ such that ${\rm H}^iM=0$ for all $i\leq a$. 
For each $i\in\Z$, let $Z^i$ be the kernel of $d^i$. 
Using truncation, we can assume $M=(0\to Z^a\to P^a\to\cdots\to P^b\to0)$. 
Consider the monomorphism
\[\xymatrix{
Z \ar@{=}[r] \ar[d]_{f} & (0 \ar[r] & Z^a \ar[r]^{=} \ar[d]^{=} & Z^a \ar[r]^{0}\ar[d] & Z^{a+1} \ar[r]^{0}\ar[d] & \cdots \ar[r]^{0} 
& Z^{b-1} \ar[r]^{0}\ar[d] & Z^b \ar[r] \ar[d]^{=} & 0) \\ 
M \ar@{=}[r] & (0 \ar[r] & Z^a \ar[r] & P^a \ar[r]^{d^a} & P^{a+1} \ar[r]^{d^{a+1}} & \cdots \ar[r]^{d^{b-2}} 
& P^{b-1} \ar[r]^{d^{b-1}} & P^b \ar[r]^{d^b} & 0)
}\]
in $\C^{\rm b}(\AA)$. 
The cokernel of $f$ in $\C^{\rm b}(\AA)$ has the form
\[B:=(0\to B^a\xrightarrow{0} B^{a+1}\xrightarrow{0}\cdots\xrightarrow{0} B^{b-1}\to0),\]
where $B^{i+1}$ is the image of $d^i$ for all $i$. 
Therefore we get a triangle 
\[Z\xrightarrow{f} M\to B\to Z[1]\]
in $\Db(\AA)$. 
We observe that $Z\cong\bigoplus_{i=a+1}^{b}Z^i[-i]$ and $B\cong\bigoplus_{i=a}^{b-1}B^i[-i]$ in $\Db(\AA)$, and that $Z^i\in\Omega^2\AA$ and $B^i\in\Omega\AA$ for all $i$.
Now the proof is completed. 
\end{proof}

\section{Functor categories}\label{section:functor category}

In this section, an upper bound of triangulated dimension is given in terms of the global dimension of a functor category. 

Let $\XX$ be an additive category. 
An $\XX$-module is an additive contravariant functor from $\XX$ to the category of abelian groups. 
A morphism between $\XX$-modules is a natural transformation. 
For any object $X\in\XX$, the functor $\Hom_\XX(-,X)$ is an $\XX$-module. 
We say that an $\XX$-module $F$ is \emph{finitely presented} if there is an exact sequence $\Hom_\XX(-,X_1)\to\Hom_\XX(-,X_0)\to F\to0$ with $X_0,X_1\in\XX$ \cite{ASS,Y}. 
The category of finitely presented $\XX$-modules is denoted by $\mod\XX$. 
The assignment $X\mapsto\Hom_\XX(-,X)$ makes a fully faithful functor $\XX\to\mod\XX$, called the \emph{Yoneda embedding} of $\XX$. 

We recall here a well-known criterion for $\mod\XX$ to be abelian. 
Let $\XX$ be an additive category and $f:X\to Y$ be a morphism in $\XX$. 
A morphism $g:Z\to X$ in $\XX$ is called a \emph{pseudo-kernel} if $\Hom_\XX(-,Z)\to\Hom_\XX(-,X)\to\Hom_\XX(-,Y)$ is exact on $\XX$. 
We say that $\XX$ \emph{has pseudo-kernels} if all morphisms in $\XX$ have pseudo-kernels.
%
The category $\mod\XX$ is an abelian category 
if and only if $\XX$ has pseudo-kernels \cite{A1}. 

Now we describe a class of additive categories having pseudo-kernels.
We say that a subcategory $\XX$ of an additive category $\AA$ is \emph{contravariantly finite} if for any object $M\in\AA$ there exist $X\in\XX$ and a morphism $f:X\to M$ such that $\Hom_\AA(X',f)$ is surjective for all $X'\in\XX$ \cite{AS}.
Dually, we define a \emph{covariantly finite} subcategory.
A covariantly and contravariantly finite subcategory is called a \emph{functorially finite} subcategory.

\begin{example}
Let $\AA$ be an additive category and $\XX$ be a contravariantly finite subcategory of $\AA$. 
If $\AA$ has pseudo-kernels, 
then $\XX$ also has pseudo-kernels.
Hence if $\AA$ is an abelian category, 
then so is $\mod\XX$. 
\end{example}

Let $\AA$ be an abelian category and $\XX$ be a subcategory of $\AA$. 
We say that $\XX$ \emph{generates} $\AA$ if for any object $M$ of $\AA$ there is an epimorphism $X\to M$ with $X\in\XX$. 

\begin{definition}
Let $\AA$ be an abelian category and $\XX$ be a contravariantly finite subcategory of $\AA$.
We denote by $\Psi:\AA\to\mod\XX$ the functor which sends $M\in\AA$ to $\Hom_\AA(-,M)$.
For any $F\in\mod\XX$, take a projective presentation 
\begin{equation}\label{projective present}
\Psi(X_1)\xrightarrow{\Psi(\alpha)}\Psi(X_0)\to F\to0
\end{equation}
of $F$ with $X_0,X_1\in\XX$. 
Then define $\Phi(F)$ by an exact sequence 
\begin{equation}\label{definition of phi}
X_1\xrightarrow{\alpha}X_0\to \Phi(F)\to0
\end{equation}
in $\AA$.
This makes a functor $\Phi:\mod\XX\to\AA$.
\end{definition}

\begin{remark}
Note that $\Phi(F)$ does not depend on the choice of projective presentations of $F$. 
Indeed, take another projective presentation
\[\Psi(Y_1)\xrightarrow{\Psi(\beta)}\Psi(Y_0)\to F\to0\]
of $F$ with $Y_0,Y_1\in\XX$. 
Since $\Psi(X_0),\Psi(X_1),\Psi(Y_0)$ and $\Psi(Y_1)$ are projective, 
there is a commutative diagram 
\[\xymatrix{
\Psi(X_1) \ar[r]_{\Psi(\alpha)}\ar[d] & \Psi(X_0) \ar[d]^{\Psi(\gamma)} \\
\Psi(Y_1) \ar[r]_{\Psi(\beta)}\ar[d] & \Psi(Y_0) \ar[d]^{\Psi(\gamma')} \\
\Psi(X_1) \ar[r]_{\Psi(\alpha)} & \Psi(X_0) 
}\]
of morphisms in $\mod\XX$. 
It is easy to see that $\Psi(1_{X_0}-\gamma'\gamma)$ factors through $\Psi(\alpha)$. 
We have a commutative diagram 
\[\xymatrix{
X_1 \ar[r]_{\alpha}\ar[d] & X_0 \ar[d]^{\gamma} \ar[r]_(0.4){\pi} & \Phi(F) \ar[r]\ar[d]^{f} & 0 \\
Y_1 \ar[r]_{\beta}\ar[d] & Y_0 \ar[d]^{\gamma'} \ar[r] & M \ar[r]\ar[d]^{g} & 0 \\
X_1 \ar[r]_{\alpha} & X_0 \ar[r]_(0.4){\pi} & \Phi(F) \ar[r] & 0 
}\]
of morphisms in $\AA$ with exact rows. 
By Yoneda's lemma, $1_{X_0}-\gamma'\gamma$ factors through $\alpha$. 
This implies $0=\pi(1_{X_0}-\gamma'\gamma)=(1_{\Phi(F)}-gf)\pi$. 
Since $\pi$ is an epimorphism, 
we observe $gf=1_{\Phi(F)}$. 
A similar argument shows that $M$ is isomorphic to $\Phi(F)$. 
\end{remark}

The following theorem is crucial to prove the main result of this section. 

\begin{theorem}\label{functor}
Let $\AA$ be an abelian category, and let $\XX$ be a contravariantly finite subcategory of $\AA$.
Then the following hold.
\begin{enumerate}[{\rm (1)}]
\item
$(\Phi,\Psi)$ is an adjoint pair.
In particular, $\Phi$ is right exact and $\Psi$ is left exact.
\item
If $\XX$ generates $\AA$, then $\Phi$ is dense and exact, and $\Phi\Psi$ is isomorphic to the identity functor.
\end{enumerate}
\end{theorem}

\begin{proof}
(1) The second statement is a general property of adjoint pairs.
Let $F\in\mod\XX$ and $M\in\AA$. 
Applying $\Hom_A(-,M)$ to \eqref{definition of phi} and $\Hom_{\mod A}(-,\Psi(M))$ to \eqref{projective present} respectively, we have a commutative diagram
\[\xymatrix{
0 \ar[r] & \Hom_\AA(\Phi(F), M) \ar[r] & \Hom_\AA(X_0, M) \ar[r]\ar[d]^{\Psi} & \Hom_\AA(X_1, M) \ar[d]^{\Psi} \\ 
0 \ar[r] & \Hom_{\mod\XX}(F, \Psi(M)) \ar[r] & \Hom_{\mod\XX}(\Psi(X_0),\Psi(M)) \ar[r] & \Hom_{\mod\XX}(\Psi(X_1),\Psi(M)) 
}\]
such that the vertical arrows are isomorphisms by Yoneda's lemma. 
Now the desired isomorphism $\Hom_\AA(\Phi(F), M)\cong\Hom_{\mod\XX}(F,\Psi(M))$ is obtained.
It shows that $(\Phi, \Psi)$ is an adjoint pair.

(2)
(i) We show that the counit $\Phi\Psi\to 1_{\AA}$ is an isomorphism if $\XX$ generates $\AA$.
Fix any $M\in\AA$.
By our assumptions on $\XX$, there exists an exact sequence $X_1\xrightarrow{\alpha}X_0\to M\to0$ with $X_0,X_1\in\XX$  such that $\Psi(X_1)\xrightarrow{\Psi(\alpha)}\Psi(X_0)\to\Psi(M)\to0$ is exact.
Then this gives the sequence \eqref{projective present} for $F:=\Psi(M)$, and it follows that $\Phi\Psi(M)\cong M$.

(ii) Let $L\xrightarrow{f}M\xrightarrow{g}N$ be a complex of objects of $\AA$. 
If the complex $\Psi(L)\xrightarrow{\Psi(f)}\Psi(M)\xrightarrow{\Psi(g)}\Psi(N)$ is exact, 
then so is $L\xrightarrow{f}M\xrightarrow{g}N$. 

Indeed, since $\Psi$ is left exact, the kernel of $\Psi(g)$ coincides with $\Psi(\Ker g)$. 
There is a morphism $h:L\to \Ker g$, and 
the assumption implies that $\Psi(h):\Psi(L)\to\Psi(\Ker g)$ is an epimorphism.
Thus any morphism from $\XX$ to $\Ker g$ factors thorugh $h$.
Since $\XX$ generates $\AA$, it follows that $h$ is an epimorphism.

(iii) We show that $\Phi$ is exact if $\XX$ generates $\AA$.  
Since $\Phi$ is right exact, we have only to prove that $\Phi$ preserves monomorphisms. 
Let $f:F\to G$ be a monomorphism in $\mod\XX$. 
Take projective resolutions \eqref{projective present} of $F$ and $\Psi(Y_2)\xrightarrow{\Psi(\gamma)}\Psi(Y_1)\xrightarrow{\Psi(\beta)}\Psi(Y_0)\to G\to0$ of $G$.
Since $\Psi(X_0)$ and $\Psi(X_1)$ are projective in $\mod\XX$, there is a commutative diagram 
\[\xymatrix{
& \Psi(X_1) \ar[r]^{\Psi(\alpha)} \ar[d]^{\Psi(h)} & \Psi(X_0) \ar[r] \ar[d]^{\Psi(g)} & F \ar[r]\ar[d]^f & 0 \\
\Psi(Y_2) \ar[r]_{\Psi(\gamma)} & \Psi(Y_1) \ar[r]_{\Psi(\beta)} & \Psi(Y_0) \ar[r] & G \ar[r] & 0
}\]
of morphisms in $\mod\XX$ with exact rows. 
Since $f$ is a monomorphism, 
the sequence
\[\Psi(X_1\oplus Y_2)\xrightarrow{\left[\begin{smallmatrix} \Psi(\alpha) & 0 \\ -\Psi(h) & \Psi(\gamma) \end{smallmatrix}\right]}
\Psi(X_0\oplus Y_1)\xrightarrow{\left[\begin{smallmatrix} \Psi(g) & \Psi(\beta) \end{smallmatrix}\right]}\Psi(Y_0)\]
is exact in $\mod\XX$. 
By (ii) we get an exact sequence 
\begin{equation}\label{exact0}
X_1\oplus Y_2\xrightarrow{\left[\begin{smallmatrix} \alpha & 0 \\ -h & \gamma \end{smallmatrix}\right]}
X_0\oplus Y_1\xrightarrow{\left[\begin{smallmatrix} g & \beta \end{smallmatrix}\right]}Y_0
\end{equation}
in $\AA$. 
There is a commutative diagram of morphisms in $\AA$
\[\xymatrix{
& X_1 \ar[r]^{\alpha} \ar[d]^{h} & X_0 \ar[r] \ar[d]^{g} & \Phi(F) \ar[r]\ar[d]^{\Phi(f)} & 0 \\
Y_2 \ar[r]_{\gamma} & Y_1 \ar[r]_{\beta} & Y_0 \ar[r] & \Phi(G) \ar[r] & 0
}\]
with exact rows, and we see by \eqref{exact0} that $\Phi(f)$ is a monomorphism.
\end{proof}

Now we can prove the main result of this section. 

\begin{theorem}\label{functor and gldim}
Let $\AA$ be an abelian category and $\XX$ a contravariantly finite subcategory that generates $\AA$. 
Then there is an inequality $\Xtridim{\XX}\Db(\AA)\leq\gldim(\mod\XX)$. 
\end{theorem}

\begin{proof}
We may assume $d:=\gldim(\mod\XX)<\infty$. 
Since $\mod\XX$ has enough projective objects with $\proj(\mod\XX)=\add\Psi(\XX)$, it holds that $\Db(\mod\XX)=\gen{\Psi(\XX)}{d+1}{}$ by Proposition \ref{KK}.
Using the dense triangle functor $\Db(\mod\XX)\to\Db(\AA)$ given in Theorem \ref{functor}, we have $\Db(\AA)=\gen{\Phi\Psi(\XX)}{d+1}{}=\gen{\XX}{d+1}{}$ by Lemma \ref{dense}.
\end{proof}

\section{Resolving subcategories}\label{section:resolving subcategory}

In this section, we give another approach to give upper bounds of dimensions of derived categories by using classical Cartan-Eilenberg resolutions. 
Our result Theorem \ref{main:resolving} gives a generalization of the inequality in Propositon \ref{KK}
due to Krause and Kussin \cite{KK}.
Note that Assadollahi and Hafezi gave a different generalization in \cite[Lemma 3.3]{AH}.

Let $\AA$ be an abelian category with enough projective objects. 
We say that a subcategory $\FF$ of $\AA$ is \emph{resolving} if
\begin{itemize}
\item
$\FF$ contains $\proj\AA$,
\item
$\FF$ is closed under direct summands, and
\item
for every exact sequence $0\to L\to M\to N\to0$ in $\AA$ with $N\in\FF$ one has $L\in\FF$ if and only if $M\in\FF$.
\end{itemize}
For a full subcategory $\CC$ of $\AA$, we denote by $\res\CC$ the smallest resolving subcategory of $\AA$ containing $\CC$.

The main result of this section is the following. 

\begin{theorem}\label{main:resolving}
Let $\AA$ be an abelian category with enough projective objects and $\FF$ a resolving subcategory.
Let $d\geq0$ and $M\in\Db(\AA)$ with $\Omega^d({\rm H}^iM)\in\FF$ for all $i\in\Z$.
Then $M$ belongs to $\gen{\FF}{\max\{2,d+1\}}{}$. 
In particular, one has $\Xtridim{\res(\Omega^d\AA)}\Db(\AA)\leq\max\{1,d\}$.
\end{theorem}
\begin{proof}
(i) Let $d\le1$. We show that $M$ is in $\gen{\FF}{2}{}$.
%
We use the same notation as in the proof of Proposition \ref{sub proj}.
Let us inductively prove that $B^i\in\FF$ and $Z^i\in\FF$ for any $i$.
Clearly, this is the case for sufficiently large $i$.

(1) We have an exact sequence $0\to Z^{i-1}\to P^{i-1}\to B^i\to0$. 
Hence $B^i\in\FF$ implies $Z^{i-1}\in\FF$. 

(2) There exists an exact sequence $0\to B^i\to Z^i\to {\rm H}^iM\to0$. 
Taking the pull-back diagram of $Z^i\to{\rm H}^iM$ and a surjection $Q_i\to{\rm H}^iM$ with $Q_i$ projective, we get an exact sequence $0\to \Omega({\rm H}^iM)\to B^i\oplus Q\to Z^i\to0$.
Since $\Omega({\rm H}^iM)$ belongs to $\FF$, it holds that $Z^i\in\FF$ implies $B^i\in\FF$.

Thus, $Z$ and $B$ belong to $\gen{\FF}{}{}$. 
Since there exists a triangle $Z\xrightarrow{f}M\to B\to Z[1]$, 
the object $M$ is in $\gen{\FF}{2}{}$.

(ii) We obtain a triangle $N\to P\to M\to N[1]$ in $\Db(\AA)$
with $P\in\gen{\proj\AA}{}{}$ and ${\rm H}^iN\cong\Omega({\rm H}^iM)$ for all $i\in\Z$,
by observing classical Cartan-Eilenberg resolutions (see also \cite[Lemma 2.5]{KK} and \cite[Proposition 3.2]{AT}).
%
%

Thus, the assertion follows inductively from (i) and (ii).
\end{proof}



\begin{example}\label{4.5}
Consider one of the following two situations.
\begin{enumerate}[{\rm (1)}]
\item
Let $R$ be a commutative Cohen-Macaulay local ring of Krull-dimension $d$, $\Lambda$ be an $R$-order  (i.e., $\Lambda$ is an $R$-algebra and $\Lambda\in\CM(R)$ \cite{A,CR}), and $\CM(\Lambda):=\{X\in\mod\Lambda \mid X\in\CM(R)\}$ be the category of (maximal) Cohen-Macaulay $\Lambda$-modules.
\item
Let $\Lambda$ be an Iwanaga-Gorenstein ring with self-injective dimension $d$ (i.e. $\Lambda$ is a noetherian ring satisfying $\id_\Lambda\Lambda=d=\id_{\Lambda^{\rm op}}\Lambda$ \cite{Iwana}), and $\CM(\Lambda):=\{X\in\mod\Lambda \mid \Ext^i_\Lambda(X,\Lambda)=0\mbox{ for any }i>0\}$ be the category of (maximal) Cohen-Macaulay $\Lambda$-modules.
\end{enumerate}
Then one has
\begin{equation}\label{CM}
\Xtridim{\CM(\Lambda)}\Db(\mod\Lambda)\leq\max\{1,d\}.
\end{equation}
In particular, if $\Lambda$ is representation-finite (i.e. $\CM(\Lambda)$ has an additive generator), then it holds that $\tridim\Db(\mod\Lambda)\leq\max\{1,d\}$.
\end{example}

Note that if $\Lambda$ is a Gorenstein $R$-order (i.e., $\Hom_R(\Lambda, \omega_R)$ is a projective $\Lambda$-module, where $\omega_R$ is a canonical module of $R$), then $\Lambda$ is an Iwanaga-Gorenstein ring and the categories $\CM(\Lambda)$ defined in Example \ref{4.5}(1) and (2) are the same.

\section{Applications}\label{section:application}

In this section, we give several applications for Iwanaga-Gorenstein rings, commutative Cohen-Macaulay rings and finitely generated commutative algebras over fields.
Throughout this section, let $\Lambda$ be a ring. 

\subsection{Cotilting modules}\label{section:cotilting module}

In representation theory, the notion of tilting modules/complexes plays an important role 
to control derived categories \cite{Ri}. 
Its dual notion of cotilting modules was studied by Auslander and Reiten as a non-commutative generalization 
of canonical modules over commutative rings \cite{AB,AM,AR2}. 
In this section, we apply the results in the previous section to rings admitting cotilting modules.
Let us begin with recalling the definition of a cotilting module. 

\begin{definition}\label{cotilting}
Let $T$ be a finitely generated $\Lambda$-module.
Denote by $\XX_T$ the subcategory of $\mod\Lambda$ consisting of modules $X$ with $\Ext_\Lambda^i(X, T)=0$ for all $i>0$.
We call $T$ \emph{cotilting} if it satisfies the following three conditions.
\begin{enumerate}[(1)]
\item
The injective dimension of the $\Lambda$-module $T$ is finite.
\item
$\Ext_\Lambda^i(T, T)=0$ for all $i>0$ (i.e., $T\in\XX_T$).
\item
For any $X\in\XX_T$, 
there exists an exact sequence $0\to X\to T'\to X'\to0$ in $\mod\Lambda$ with $T'\in\add T$ and $X'\in\XX_T$.
\end{enumerate}
\end{definition}

\begin{example}
\begin{enumerate}[(1)]
\item
Let $R$ be a commutative Cohen-Macaulay local ring with a canonical module $\omega_R$. 
Then $\omega_R$ is a cotilting module over $R$ and $\XX_{\omega_R}=\CM(R)$ holds.
Let $\Lambda$ be an $R$-order.
For any tilting $\Lambda^{\rm op}$-module $T$ in the sense of Miyashita \cite{Mi} with $T\in\CM(R)$, the $\Lambda$-module $\Hom_R(T,\omega_R)$ is cotilting.
For the cotilting $\Lambda$-module $\omega_\Lambda:=\Hom_R(\Lambda,\omega_R)$ it holds that $\XX_{\omega_\Lambda}=\CM(\Lambda)$.
\item
Let $\Lambda$ be an Iwanaga-Gorenstein ring.
Then $\Lambda$ is a cotilting module over $\Lambda$, and hence $\XX_\Lambda=\CM(\Lambda)$. 
\end{enumerate}
\end{example}

Let $T$ be a cotilting module over $\Lambda$. 
For any $M\in\mod\Lambda$, 
we define $d_M$ to be the minimum integer $n$ satisfying $\Ext_\Lambda^i(M,T)=0$ for any $i>n$. 
Clearly, $d_M\le\id T$. 
%
%
%

Now the result below is obtained.

\begin{theorem}\label{cotilting and dimension}
Let $T$ be a cotilting module over $\Lambda$. 
Then one has 
\[\Xtridim{\XX_T}\Db(\mod\Lambda)\leq\max\{1, \id T\}.\]
\end{theorem}
\begin{proof}
It is known by Auslander-Buchweitz approximation theory \cite{AB}
that $\XX_T$ is a contravariantly finite subcategory of $\mod\Lambda$ generating $\mod\Lambda$.
As a special case of \cite[(2.6.1)]{I2}, we have the inequality $\gldim(\mod\XX_T)\leq\max\{2,\id T\}$.
Thus, the inequality $\Xtridim{\XX_T}\Db(\mod\Lambda)\leq\max\{2,\id T\}$ holds by Theorem \ref{functor and gldim}. 
If $\id T\le1$, then $\Omega(\mod\Lambda)$ is contained in $\XX_T$, and $\Xtridim{\XX_T}\Db(\mod\Lambda)\leq1$
by Proposition \ref{sub proj}. 
\end{proof}

\begin{remark}
If $\id T=0$, then the inequality $\Xtridim{\XX_T}\Db(\mod\Lambda)\leq\id T$ does not necessarily hold. 
Let $\Lambda$ be a finite dimensional non-semisimple self-injective algebra over a field. 
Then the $\Lambda$-module $\Lambda$ is a cotilting module with $\id\Lambda=0$ and $\XX_\Lambda=\mod\Lambda$. 
However, $\gen{\mod\Lambda}{}{}$ is different from $\Db(\mod\Lambda)$.
\end{remark}

We close this subsection by giving more general upper bounds for the derived dimensions.
Let $\Lambda$ be a ring and $T$ a cotilting $\Lambda$-module.
Recall that a subcategory $\CC$ of $\XX_T$ is \emph{$n$-cluster tilting} 
 (or \emph{maximal $(n-1)$-orthogonal}) \cite{I1,I2} 
if $C$ is a functorially finite subcategory of $X_T$ and 
\begin{eqnarray*}
\CC&=&\{X\in\XX_T\ |\ \Ext_\Lambda^i(\CC,X)=0\ \mbox{ for any }\ 0<i<n\}\\
&=&\{X\in\XX_T\ |\ \Ext_\Lambda^i(X,\CC)=0\ \mbox{ for any }\ 0<i<n\}.
\end{eqnarray*}
The following result holds, where the case $n=1$ is an analogue of Theorem \ref{cotilting and dimension}.

\begin{proposition}
There is an inequality $\Xtridim{\CC}\Db(\mod\Lambda)\le\max\{n+1,\id T\}$.
\end{proposition}
\begin{proof}
It is shown in \cite[(2.6.1)]{I2} that $\gldim(\mod\CC)\le\max\{n+1,\id T\}$.
Thus the assertion immediately follows from Theorem \ref{functor and gldim}.
\end{proof}

\subsection{Krull dimension in terms of triagle dimension}

The purpose of this subsection is to show that our triangle dimension with respect to a subcategory gives a formulation of the Krull dimensions of rings in terms of the triangulated structure of their derived categories.

Let $R$ be a commutative ring, and let $\Lambda$ be an $R$-algebra which is a finitely generated faithful $R$-module.
We define the \emph{singular locus} by
\[\operatorname{Sing}_R\Lambda:=\{\p\in\Spec R\ |\ \gldim{\Lambda_\p}>\dim R_\p\}.\]
When $\Lambda=R$, this gives the usual singular locus $\operatorname{Sing}R$ of $R$, that is, the set of prime ideals $\p$ of $R$ such that the local ring $R_\p$ is not regular.
For a subset $W$ of $\Spec R$, let $\operatorname{NP}^{-1}(W)$ be the subcategory of $\mod\Lambda$ consisting of modules $X$ such that $X_\p\in\proj\Lambda_\p$ for all $\p\in\Spec R\setminus W$.
We start by making a lemma.

\begin{lemma}\label{NP}
Let $R$ be a commutative ring and and $\Lambda$ be an $R$-algebra that is a finitely generated $R$-module.
Then the following hold.
\begin{enumerate}[{\rm (1)}]
\item
$\Xtridim{\operatorname{NP}^{-1}(\operatorname{Sing}_R\Lambda)}\Db(\mod\Lambda)\le\max\{1,\dim R\}$.
\item
$\Xtridim{\operatorname{NP}^{-1}(\Spec R\setminus\{\p\})}\Db(\mod\Lambda)\ge\gldim\Lambda_\p$ for every $\p\in\Spec R$.
\end{enumerate}
\end{lemma}
\begin{proof}
(1) Let $d=\dim R$.
For a prime ideal $\p$ of $R$ which is not in $\operatorname{Sing}_R\Lambda$, 
we have $(\Omega^d(\mod\Lambda))_\p\subseteq\Omega^d(\mod\Lambda_\p)\subseteq\proj\Lambda_\p$, as $\gldim\Lambda_\p=\dim R_\p\leq d$.
Hence $\Omega^d(\mod\Lambda)$ is contained in $\operatorname{NP}^{-1}(\operatorname{Sing}_R\Lambda)$, 
and the assertion follows from Theorem \ref{main:resolving}, 
since $\operatorname{NP}^{-1}(\operatorname{Sing}_R\Lambda)$ is a resolving subcategory of $\mod\Lambda$.

(2) Let $d=\gldim\Lambda_\p$. 
Assume that the left hand side is at most $d-1$. 
Then we have $\Db(\mod\Lambda)=\gen{\operatorname{NP}^{-1}(\Spec R\setminus\{\p\})}{d}{}$.
Localizing this at $\p$, we see that $\Db(\mod\Lambda_\p)=\gen{\operatorname{NP}^{-1}(\Spec R\setminus\{\p\})_\p}{d}{}\subseteq\gen{\proj\Lambda_\p}{d}{}$.
This contradicts Proposition \ref{KK}.
\end{proof}

The proposition below immediately follows from Lemma \ref{NP}.

\begin{proposition}\label{affine}
Let $R$ be a commutative ring of positive Krull dimension.
Let $\Lambda$ be an $R$-algebra that is a finitely generated $R$-module.
Assume that there exists a maximal ideal $\m$ of $R$ such that $\gldim\Lambda_\m=\dim R_\m=\dim R$.
Then one has
\[\dim R=\Xtridim{\operatorname{NP}^{-1}(\operatorname{Sing}_R\Lambda)}\Db(\mod\Lambda).\] 
\end{proposition}

\begin{proof}
Since $\operatorname{Sing}_R\Lambda$ is contained in $\Spec R\setminus\{\m\}$, we have
$$
\Xtridim{\operatorname{NP}^{-1}(\operatorname{Sing}_R\Lambda)}\Db(\mod\Lambda)\ge\gldim\Lambda_{\m}=\dim R
$$
by Lemma \ref{NP}(2).
The converse inequality follows from Lemma \ref{NP}(1).
\end{proof}

Now we can realize the goal for this subsection.

\begin{theorem}\label{3122219}
Let $R$ be a finitely generated commutative algebra over a field. 
Suppose that $R_\p$ is a field for some $\p\in\Spec R$ with $\dim R/\p=\dim R$ 
(e.g. $R$ is reduced). 
Then there is an equality 
\[\dim R=\Xtridim{\operatorname{NP}^{-1}(\operatorname{Sing}R)}\Db(\mod R).\]
\end{theorem}
\begin{proof}
We find an ideal $J$ of $R$ such that $\operatorname{Sing}R=V(J)$. 
By assumption, $(J+\p)/\p$ is a nonzero ideal of the integral domain $R/\p$ that is finitely generated over a field.
Hence there exists a maximal ideal $\m/\p$ of $R/\p$ which does not contain $(J+\p)/\p$. 
Note that $\dim R_\m=\dim R$ and that $R_\m$ is regular. 
The equality in the theorem follows from Proposition \ref{affine}.
\end{proof}

We should remark that the right hand side of the equality in the Theorem \ref{3122219} is independent of the Krull dimension of $R$.
So, in the case where $R$ is a finitely generated reduced commutative algebra over a field, the Krull dimension of $R$ is determined by the triangulated structure of $\Db(\mod R)$ and the singular locus of $R$.

\subsection{Derived dimensions and stable dimensions}

In this subsection, we investigate the relationship between derived dimensions and stable dimensions.

Let $\Lambda$ be an Iwanaga-Gorenstein ring.
Let $\sCM(\Lambda)$ denote the {\em stable category} of $\CM(\Lambda)$.
This category is defined as follows: the objects of $\sCM(\Lambda)$ are the same as those of $\CM(\Lambda)$, and the hom-set $\Hom_{\sCM(\Lambda)}(M,N)$ is the quotient group of $\Hom_\Lambda(M,N)$ by the subgroup consisting of morphisms factoring through projective modules. This is usually denoted by $\underline{\Hom}_\Lambda(M,N)$.
It is known that $\sCM(\Lambda)$ is a triangulated category \cite{B,H}, and its triangle dimension $\tridim\sCM(\Lambda)$ is called the {\em stable dimension} of $\Lambda$.
Since there is a dense triangle functor $\Db(\mod\Lambda)\to\sCM(\Lambda)$ \cite{B}, the stable dimension is bounded above by the derived dimension:
$$
\tridim\sCM(\Lambda)\le\tridim\Db(\mod\Lambda)
$$
holds by Lemma \ref{dense}.

Let $\CC$ be a subcategory of $\CM(\Lambda)$ and $n$ a positive integer.
Then we denote by $\Omega^{-n}\CC$ the subcategory of $\CM(\Lambda)$ consisting of objects $X$ that admits an exact sequence $0\to C\to P^0\to\cdots\to P^{n-1}\to X\to0$ in $\CM(\Lambda)$ with each $P^i$ projective and $C\in\CC$.
Note that $\Omega^{-n}\CC$ contains $\proj\Lambda$.

\begin{lemma}\label{scm and derived}
Let $\Lambda$ be an Iwanaga-Gorenstein ring of self-injective dimension $d$. 
If there exists a subcategory $\XX$ of $\CM(\Lambda)$ 
such that $\sCM(\Lambda)=\gen{\XX}{n+1}{}$ for some non-negative integer $n$, 
then we have 
\[\Db(\mod\Lambda)=\gen{\XX^\Omega}{\max\{2, d+1\}\cdot(n+1)}{}\]
where $\XX^\Omega=\bigcup_{i\in\Z}\Omega^i\XX$. 
\end{lemma}

\begin{proof}
In view of \eqref{CM}, it suffices to show that if $M\in\CM(\Lambda)$ belongs to $\gen{\XX}{m+1}{\sCM(\Lambda)}$, then it is in  $\gen{\XX^\Omega}{m+1}{\Db(\mod\Lambda)}$. 
We use induction on $m$. 
It is clear when $m=0$, so assume $m>0$. 
Then there is a triangle 
$$
X\to M\oplus M'\to Y\to X[1]
$$
in $\sCM(\Lambda)$ with $X\in\gen{\XX}{}{\sCM(\Lambda)}$ and $Y\in\gen{\XX}{m}{\sCM(\Lambda)}$.
This gives rise to an exact sequence
$$
0\to X\to M\oplus M'\oplus P\to Y\to0
$$
in $\mod\Lambda$ for some $P\in\proj\Lambda$. 
The basis and hypothesis of the induction show 
that $X\in\gen{\XX^\Omega}{}{\Db(\mod\Lambda)}$ and $Y\in\gen{\XX^\Omega}{m}{\Db(\mod\Lambda)}$. 
It follows that $M$ is in $\gen{\XX^\Omega}{m+1}{\Db(\mod\Lambda)}$. 
\end{proof}

The following inequality is obtained by using Lemma \ref{scm and derived}.

\begin{proposition}\label{hypersurface and derived dimension}
Let $R$ be a $d$-dimensional hypersurface. 
Then one has 
\[\frac{\tridim\Db(\mod R)+1}{\tridim\underline{\CM}(R)+1}\leq\max\{2, d+1\}.\] 
\end{proposition}
\begin{proof}
Putting $n=\tridim\sCM(R)$, we have $\sCM(R)=\gen{G}{n+1}{\sCM(R)}$ for some $G\in\CM(R)$. 
Since $R$ is a hypersurface, 
any $M\in\CM(R)$ is isomorphic to $\Omega^2 M$ in $\sCM(R)$ (cf. \cite{Y}). 
It now follows from Lemma \ref{scm and derived} 
that
$$
\Db(\mod R)=\gen{G\oplus\Omega G\oplus R}{\max\{2, d+1\}\cdot(n+1)}{\Db(\mod R)},
$$
which yields the inequality in the proposition.
\end{proof}

\begin{example}
Let $k$ be a field of characteristic zero and let $R=k[[x_0,\cdots,x_d]]/(f)$ with $f\in(x_0,\cdots,x_d)^2$. 
Then one has
$$
\tridim\Db(\mod R)\leq 2\max\{2,d+1\}\LoL(R/(\partial f))-1,
$$
where $(\partial f)=(\frac{\partial f}{\partial x_0},\cdots,\frac{\partial f}{\partial x_d})R$ is the Jacobian ideal (see also the next subsection) and $\LoL(-)$ denotes the Loewy length. 

Indeed, by \cite[Proposition 4.11]{BFK1}, 
we have $\tridim\sCM(R)\leq2\LoL(R/(\partial f))-1$. 
Hence the assertion follows from Proposition \ref{hypersurface and derived dimension}. 
\end{example}

\begin{example}
Let $R$ be a $d$-dimensional complete local hypersurface 
over an algebraically closed field of characteristic not two. 
Suppose that $R$ has countable Cohen-Macaulay representation type (i.e., there exist countably many isomorphism classes of indecomposable Cohen-Macaulay $R$-modules). 
Then there exists $G\in\CM(R)$ such that $\Db(\mod R)=\gen{G\oplus \Omega G\oplus R}{2\cdot\max\{2, d+1\}}{}$. 
Hence one has
$$
\tridim\Db(\mod R)\leq\max\{3, 2d+1\}.
$$

In fact, by \cite[Theorem 1.1]{AIT} there exists $G\in\CM(R)$ such that any Cohen-Macaulay $R$-module $X$ admits an exact sequence $0\to G_1\to X\oplus R'\to G_2\to0$ with $G_1,G_2\in\{G, \Omega G\}$ and $R'\in\proj R$. 
Therefore $\CM(R)$ is contained in $\gen{G\oplus \Omega G\oplus R}{2}{\Db(\mod R)}$, and the inequality is obtained by \eqref{CM}. 
\end{example}

\subsection{Upper bounds by Noether differents}

In the recent work \cite{sing}, Dao and Takahashi have obtained some upper bounds for the triangle dimension of the singularity category of a commutative Cohen-Macaulay local ring with an isolated singularity, by using Noether differents and Jacobian ideals.
In this subsection, we obtain similar upper bounds for the derived dimensions,
taking advantage of our inequality \eqref{CM}.

Let $(R,\mathfrak{m},k)$ be a $d$-dimensional commutative complete local ring containing a field.
Let $A$ be a Noether normalization of $R$, that is, a formal power series subring $k[[x_1,\dots,x_d]]$, 
where $x_1,\dots,x_d$ is a system of parameters of $R$.
Let $R^e=R\otimes_AR$ be the enveloping algebra of $R$ over $A$.
Define a map $\mu:R^e\to R$ by $\mu(a\otimes b)=ab$ for $a,b\in R$.
Then the ideal $\mathfrak{N}^R_A=\mu(\operatorname{Ann}_{R^e}\operatorname{Ker}\mu)$ of $R$ is called 
the {\em Noether different} of $R$ over $A$.
We denote by $\mathfrak{N}^R$ the sum of $\mathfrak{N}^R_A$, where $A$ runs through the Noether normalizations of $R$.

Cohen's structure theorem yields an isomorphism
$$
R\cong k[[x_1,\dots,x_n]]/(f_1,\dots,f_m).
$$
Identifying $R$ with this quotient ring, we define the {\em Jacobian ideal} $J$ 
as the ideal of $R$ generated by the $h$-minors of the Jacobian matrix $(\frac{\partial f_i}{\partial x_j})$, 
where $h$ is the height of the ideal $(f_1,\dots,f_m)$ of $k[[x_1,\dots,x_n]]$.

For an $\mathfrak m$-primary ideal $I$ of $R$, we denote by $\mathrm{e}(I)$ the {\em multiplicity} of $I$, 
i.e., $\mathrm{e}(I)=\lim_{n\to\infty}\frac{d!}{n^d}\ell_R(R/I^{n+1})$.
Recall that $R$ is said to be an {\em isolated singularity} if the local ring $R_{\mathfrak p}$ is regular 
for every nonmaximal prime ideal $\mathfrak p$ of $R$.

\begin{proposition}
Let $(R,\mathfrak{m},k)$ be a $d$-dimensional Cohen-Macaulay complete local ring containing a field, with $k$ infinite and perfect.
Suppose that $R$ is an isolated singularity.
Then one has inequalities
$$
\tridim\Db(\mod R)<\max\{2,d+1\}\cdot\mathrm{e}(\mathfrak{N}^R)\le\max\{2,d+1\}\cdot\mathrm{e}(J).
$$
\end{proposition}

\begin{proof}
We see from \cite[Lemmas 4.3, 5.8 and Propositions 4.4, 4.5]{Wang} and \cite[Lemma (6.12)]{Y} 
that $\mathfrak{N}^R$ and $J$ are $\mathfrak m$-primary ideals with $J\subseteq\mathfrak{N}^R$.
The second inequality immediately follows from this.

Let us prove the first inequality.
Since $k$ is an infinite field, we can choose a parameter ideal $Q$ of $R$ contained in $\mathfrak{N}^R$ 
such that $\mathrm{e}(\mathfrak{N}^R)=\ell_R(R/Q)$; see \cite[Corollary 4.6.10 and Theorems 4.7.6, 4.7.10]{BH}.
Let $M$ be a Cohen-Macaulay $R$-module.
Using \cite[Corollary 5.13]{Wang}, we have $Q\mathrm{Ext}_R^i(M,\Omega^jM)=0$ for $1\le i,j\le d$.
By \cite[Proposition 2.2]{stcm} the module $M$ is isomorphic to a direct summand of $\Omega^d(M/QM)$.
Note that $R/Q$ is an artinian ring.
Consider the composition series of $R/Q$:
$$
R/Q=I_0\supsetneq I_1\supsetneq \cdots \supsetneq I_l=0,
$$
where $l=\ell_R(R/Q)$.
There is a filtration
$$
M/QM=I_0(M/QM)\supseteq I_1(M/QM)\supseteq \cdots \supseteq I_l(M/QM)=0
$$
whose subquotients are $k$-vector spaces.
Decomposing this into short exact sequences and taking the $d$-th syzygies, 
we observe that $M$ belongs to $\langle\Omega^dk\rangle_l^{\mathsf{D}^b(\operatorname{mod}R)}$.
Hence we obtain $\mathrm{CM}(R)\subseteq\langle\Omega^dk\rangle_l^{\mathsf{D}^b(\operatorname{mod}R)}$, 
and \eqref{CM} yields equalities $\mathsf{D}^b(\operatorname{mod}R)
=\langle\Omega^dk\rangle_{\max\{2,d+1\}\cdot l}=\langle\Omega^dk\rangle_{\max\{2,d+1\}\cdot\mathrm{e}(\mathfrak{N}^R)}$.
\end{proof}


\end{document}